\documentclass[12pt,a4paper]{amsart}
\usepackage[utf8]{inputenc}
\usepackage{amsmath,amssymb,amsfonts}
\usepackage{graphicx,epsfig}
\usepackage{amsfonts}
\usepackage{mathrsfs}
\usepackage{amssymb}
\usepackage{amsmath}
\usepackage{amsthm}
\usepackage{cite}
\newtheorem{theorem}{Theorem}[section]
\newtheorem{definition}[theorem]{Definition}
\newtheorem{lemma}[theorem]{Lemma}
\newtheorem{corollary}[theorem]{Corollary}
\newtheorem{remark}{Remark}

\renewcommand{\Re}{{\,\operatorname{Re}\,}}

\begin{document}

\title{Coefficient estimates for the $q$-starlike class}
\author{Ming Li$^{*}$, Ao-Li Zhu}

\begin{abstract}
By employing the $q$-difference operator, various classes of $q$-extensions of starlike functions have emerged from many different viewpoints and perspectives. In this paper, we improve some results for the Bieberbach-type problems of $q$-starlike class by Agrawal [S. Agrawal, Indagationes Mathematicae, 2021] and provide a comprehensive analysis of the Taylor coefficients of $q$-starlike functions using
the Carath\'eodory-Toeplitz theorem and its corollaries. For real $q\in(0,1)$, we obtain precise upper bounds for the initial few coefficients. Subsequently, we analyze the Hankel and Toeplitz determinants based on the foundational results. Finally, we establish the upper bounds of the coefficients by applying Parseval's theorem to $q$-starlike functions with complex $q$.
\end{abstract}
\maketitle

\renewcommand{\thefootnote}{}
\footnotetext{\ *Corresponding author}

\vspace{-1cm}

\section{Introduction}
Let $\mathbb{C}$ denote the complex plane and $\mathbb{D}=\{z:|z|<1\}$ be the unit disc. Denote $\mathcal{H}$ the family of all analytic functions in $\mathbb{D}$. $\mathcal{A}\subset\mathcal{H}$ is the family of functions have the form of
\begin{equation}\label{A}
f(z)=z+a_2z^2+a_3z^3+\cdots,\quad\text{for}\;z\in\mathbb{D}.
\end{equation}
Recall that a set \( E \subset \mathbb{C} \) is said to be starlike with respect to a point \( \omega_0 \in E \) if the line segment joining \( \omega_0 \) to every other point \( \omega \in E \) lies entirely within \( E \). \( S^* \) denotes the class of functions that map the unit disk to a region starlike with respect to 0. A function \( f \in S^* \) can be analytically characterized as follows: \( f \in \mathcal{A} \) and satisfies \( \Re\left( \frac{zf'(z)}{f(z)} \right) > 0 \) for all \( z \in \mathbb{D} \).
From the perspective of series, it is straightforward to observe that the derivative of a function \( f \in \mathcal{A} \) can be re-expressed using convolution (denoted by ``$\ast$") as follows:
\begin{equation}
f'(z)=\frac{1}{z}\left\{f(z)\ast\frac{z}{(1-z)^2}\right\}.
\end{equation}
Here, the convolution is defined in the sense of Ruscheweyh (see, e.g., \cite{Ruscheweyh1982}) as
\[
(f\ast g)(z):=z+a_2b_2z^2+a_3b_3z^3+\cdots,
\]
where \( f(z)=z+\sum_{n=2}^{\infty}a_nz^n \) and \( g(z)=z+\sum_{n=2}^{\infty}b_nz^n \) are elements of \( \mathcal{A} \).
Recently, Piejko et al. \cite{piejko2019q,piejko2020convolution} introduced another operator
\begin{equation}
D_{\zeta}f(z)=\frac{1}{z}\left\{f(z)\ast\frac{z}{(1-\zeta z)(1-z)}\right\},
\end{equation}
where $\zeta\in \mathbb{C}$ with $|\zeta|\leq 1$.
Consequently, the authors provided the definition of $\zeta$-starlike functions \cite{piejko2023}, which is defined as following.
\begin{definition}\cite{piejko2020convolution,piejko2023}
	Let  $\zeta\in \mathbb{C}$ with $|\zeta|\leq 1$. An analytic function $f\in\mathcal{A}$  is said to be $\zeta$-starlike of order $\alpha$ if it satisfies the condition
	\begin{equation}\label{def}
	\Re\frac{zD_{\zeta}f(z)}{f(z)}>\alpha,\quad z\in\mathbb{D},
	\end{equation}
	 where $\alpha\in[0,1)$. The class of all such functions is denoted by $S_{\zeta}^*(\alpha)$.
\end{definition}
It is worth noting that if we take $\zeta=q$ and $q\in(0,1)$, $D_{\zeta}$ reduces to the $q$-derivative \cite{Jackson1909}, which is defined by
\begin{equation*}
D_qf(z)=\left\{\begin{aligned}
&\frac{f(z)-f(qz)}{(1-q)z},\quad &&z\neq 0,\\
&f'(0),\quad &&z=0.
\end{aligned}
\right.
\end{equation*}
It is clear that $\displaystyle\lim_{q\to 1^-}D_qf(z)=f'(z)$. 
By using the $q$-derivative, various families of $q$-extensions of starlike functions have emerged. For example, Ismail et al. \cite{ismail1990} introduced a class of generalized starlike functions:
$
PS_q^*=\left\{f\in \mathcal{A}:\left|\frac{zD_qf(z)}{f(z)}-\frac{1}{1-q}\right|\leq \frac{1}{1-q}\right\}.
$

Agrawal \cite{Agrawal2021} investigated the class of $q$-starlike functions of order $\alpha$, which constitutes a particular instance of $\zeta$-starlike functions when $\zeta=q$. Throughout this paper, we will use the simplifed notation  $S_{q}^*$ instead of $S_{q}^*(0)$.

The Hankel and Toeplitz determinants (defined in Section 4) play an important role in the study of singularities, power series with integral coefficients, and many other areas \cite{grenander1958,pommerenke1966,annals2011}.
The upper bounds of them for starlike functions have been thoroughly examined by various researchers \cite{janteng2007,choo2018,Kwon2019H,zaprawa2021,kowalczyk2022sharp,rath2022,sim2022,wang2023,Verma2023,Ali2023,zaprawa2024,Kumarverma2025,zaprawa2025}.
Related studies on \( q \)-starlike functions have appeared separately in \cite{Agrawal2021,srivastava2021,Taj2025,Qiu2025}.  

In this paper, we systematically investigates the coefficient properties and Hankel determinants of analytic function family $S_{\zeta}^*(\alpha)$. In section 3, we focus on the Bieberbach-type problem to the family $S_q^*$ (where $q$ is a real number). We provide upper bounds of $|a_n|$ for $f\in S_{q}^*$. The precise bounds of coefficients $|a_2|,|a_3|,|a_4|$ for $f\in S_{q}^*$ are provided. In section 4, we examine the Hankel and Toeplitz determinants for the coefficients of functions in the family $S_q^*$, not only significantly improving Zaprawa's results [S. Agrawal, Indagationes Mathematicae, 2021, Theorem 3.4] on the upper bounds of second-order Hankel determinants but also extending the research methodology to the estimation problem of higher-order determinants. Notably, in section 5, we expand the research perspective to the complex domain, conducting an in-depth exploration of the Bieberbach-type problem to family $S_{\zeta}^*(\alpha)$ (with parameter $\zeta\in\mathbb{C}$ and $|\zeta|\leq 1$) and obtaining a series of new findings with theoretical significance.

\section{Preliminaries}

The Carath\'{e}odory class $\mathcal {P} $ and Schwarz class play important roles in geometric function theory.
They are defined by
\begin{align*}
\mathcal P:=\{p(z)\in\mathcal{H}: p(0)=1, \Re p(z)>0\}\;\text{and}\;
\mathcal{B}_0:=\{\omega\left(z\right)\in\mathcal{H}: \omega(0)=0, \left|\omega\left(z\right)\right|<1\}
\end{align*}
separately. It is well known that for any function $p(z)=1+p_1z+p_2z^2+p_3z^3+\cdots\in\mathcal{P}$, there is always an analytic function
$\omega\left(z\right)=b_1z+b_2z^2+\cdots\in\mathcal{B}_0$
such that
\begin{equation}\label{pw}
p(z)=\frac{1+\omega(z)}{1-\omega(z)}.
\end{equation}
A comparison of the coefficients on both sides of the equation \eqref{pw} reveals the following relationships
\begin{equation}\label{b1b2b3}
2b_1=p_1,\quad 4b_2=2p_2-p_1^2,\quad 8b_3=4p_3-4p_1p_2+p_1^3.
\end{equation}

A necessary and sufficient condition for \( \mathcal{P} \) is given by the Carath\'{e}odory-Toeplitz Theorem (see \cite{grenander1958} or \cite{tsuji1959}).

Based on this theorem, Libera and Z{\l}otkiewicz \cite{libera1982} parameterized the coefficients of the Carath\'{e}odory function as follows:
\begin{equation}\label{p2}
2p_2 = p_1^2 + x(4 - p_1^2),
\end{equation}
\begin{equation}\label{p3}
4p_3 = p_1^3 + 2(4 - p_1^2)p_1x - p_1(4 - p_1^2)x^2 + 2(4 - p_1^2)(1 - |x|^2)y,
\end{equation}
under the assumption that \( p_1 > 0 \). Here, \( x \) and \( y \) are complex numbers satisfying \( |x| \leq 1 \) and \( |y| \leq 1 \), respectively.
Applying the representations of $p_2$ and $p_3$ from \eqref{p2} and \eqref{p3} to \eqref{b1b2b3} gives the following.
\begin{lemma}\label{b1}\cite{Simon2005}
	For a function $\omega(z)=b_1z+b_2z^2+\cdots\in\mathcal{B}_0$, if $b_1>0$, then
	there are complex numbers $x$, $y$ such that
	\begin{equation}\label{b2b3}
	b_2=x(1-b_1^2),\quad b_3=(1-b_1^2)[(1-|x|^2)y-b_1x^2],
	\end{equation}
	and $|x|\leq 1$, $|y|\leq 1$.
\end{lemma}
\begin{remark}\label{wx-1}
When \( b_1 \) is a real number, \( p_1 \) is also real. Furthermore, if we set \( x = -1 \) in \eqref{p2} (i.e., \( p_2 = p_1^2 - 2 \)), then by a lemma due to Li and Sugawa (Lemma 2.3 in \cite{li2017note}), we obtain 
\(
p(z) = (1 - z^2)/(1 - p_1 z + z^2).
\)
Consequently, from \eqref{pw} it follows that either 
\(
\omega(z) = z(2z - p_1)/(p_1 z - 2)
\)
or 
\(
\omega(z) = z(z - b_1)/(b_1 z - 1).
\)
\end{remark}
\begin{remark}\label{wb11}
Observe that when $b_1=1$, we have $b_2=0$ and $b_3=0$. Schwarz lemma states that the only function in $\mathcal{B}_0$ with $b_1=1$ is $\omega(z)=z$.
\end{remark}
This paper concerns two fundamental functionals for class $\mathcal{B}_0$, with the primary result established by Prokhorov and Szynal \cite{prokhorov1981}.

\begin{lemma}\label{b2}\cite{prokhorov1981}
	If an analytic function \( \omega(z) = b_1 z + b_2 z^2 + \cdots \in \mathcal{B}_0 \), then for all real numbers \( \mu, \nu \), the following sharp estimate holds:
	\[
	|b_3 + \mu b_1 b_2 + \nu b_1^3| \leq |\nu| \quad \text{for} \quad (\mu, \nu) \in D_1
	\]
	where
	\begin{equation*}
	\begin{aligned}
	&D_1=\left\{(\mu,\nu):|\mu|\geq \frac{1}{2}, \,\nu\leq -\frac{2}{3}(|\mu|+1)\; \text{or}\; |\mu|\geq 4, \,\nu\geq \frac{2}{3}(|\mu|-1)\right\}.
	\end{aligned}
	\end{equation*}
\end{lemma}

The second functional is
\begin{equation}\label{Y}
Y(a,b,c)=\max_{z\in \mathbb{D}}(|a+bz+cz^2|+1-|z|^2),
\end{equation}
which is central to our proofs and connects to Ohno-Sugawa's \cite{ohno2018} findings.
\begin{lemma}\label{OhnoSugawa}\cite{ohno2018}
	Let $a,b,c\in {\mathbb R}$, if $ac\geq 0$, then
	\begin{equation*}
	Y(a,b,c)=
	\left\{
	\begin{aligned}
	&|a|+|b|+|c|,\qquad \qquad \left(|b|\geq 2(1-|c|)\right), \\
	&1+|a|+\frac{b^2}{4(1-|c|)},\quad \; \left(|b|<2(1-|c|)\right).
	\end{aligned}
	\right.
	\end{equation*}
\end{lemma}
As a special case, we have the following lemma.
\begin{lemma}\label{lemmaY2}
For $ q\in (0,1)$ and $b_1\in(0,1)$, if $a$, $b$ and $c$  are given by 
	\begin{equation}\label{abc}
	a=\frac{-(2+q)b_1^3}{(1-b_1^2)(1+q)^2},\quad b=\frac {2b_1}{(1+q)^2},\quad c=-b_1-\frac{(1+q+q^2)(1-b_1^2)(1+q)^2}{b_1},
	\end{equation}
then
	$$Y(a,b,c)=\frac{(1-q)b_1}{(1+q)(1-b_1^2)}+\frac{1+q+q^2}{b_1(1-b_1^2)(1+q)^2}.$$
\end{lemma}
\begin{proof}
First, from equations \eqref{abc}, it is straightforward to observe that \( a < 0 \), \( b > 0 \), and \( c < 0 \) given \( q \in (0,1) \) and \( b_1 \in (0,1) \). Thus, we have \( ac > 0 \). Next, consider the quantity
\[
|b| - 2(1 - |c|) = b - 2(1 + c) = \frac{(2 + 2q)b_1^2 - 2(1 + q)^2b_1 + 2(1 + q + q^2)}{(1 + q)^2b_1}.
\]
Let the numerator be defined as
\[
(2 + 2q)b_1^2 - 2(1 + q)^2b_1 + 2(1 + q + q^2) =: H(b_1),
\]
which is a quadratic function in \( b_1 \). A simple calculation shows that its discriminant is
\[
\Delta = 4(1 + q)^4 - 16(1 + q + q^2)(1 + q) = 4q^4 - 8q^2 - 16q - 12.
\]
It is easy to verify that \( \Delta < 0 \) for \( q \in (0,1) \), since
\[
4q^4 - 8q^2 - 16q - 12 < 4 - 8q^2 - 16q - 12 = -8(q + 1)^2 < 0.
\]
Given \( 2 + 2q > 0 \), we conclude \( H(b_1) > 0 \), and therefore \( |b| - 2(1 - |c|) > 0 \) for \( b_1 \in (0,1) \). Applying Lemma \ref{OhnoSugawa}, it follows that
\[
Y(a,b,c) = |a| + |b| + |c| = -a + b - c = \frac{(1 - q)b_1}{(1 + q)(1 - b_1^2)} + \frac{1 + q + q^2}{b_1(1 - b_1^2)(1 + q)^2}.
\]
This completes the proof.
\end{proof}

\section{The Bieberbach-type problems on $S_q^*$}
In \cite{Agrawal2021}, Agrawal established the coefficient bound $|a_n|\leq \displaystyle\prod_{k=2}^{n}\left|\dfrac{1+[k-1]_q}{[k]_q-1}\right|$ for $f \in S_{q}^*$. In this section, we derive sharp estimates for the initial coefficients.
\begin{theorem}\label{theorem:1}
	For $f(z)=z+\displaystyle\sum_{n=2}^{\infty}a_nz^n\in S_{q}^*$, $q\in(0,1)$, we have
	$$|a_2|\leq \frac{2}{q},\qquad |a_3|\leq\frac{4+2q}{q^2(1+q)},\quad |a_4|\leq \frac{2(4+4q+3q^2+q^3)}{q^3(1+q+q^2)(1+q)}.$$
	These estimates are sharp, with extremal function 
	\begin{equation}\label{solution}
	f(z) = z \prod_{k=0}^\infty \frac{q - q^{k+1} z}{q- q^k (2-q)z}.
	\end{equation}
\end{theorem}
\begin{proof}
By the definition of $q$-starlike function and the relationship between Carath\'eodory functions and Schwartz functions, for a function $f(z)=z+a_2z^2+a_3z^3+\cdots\in S^*_{q} $, there exists an analytic function
	$\omega(z)=b_1z+b_2z^2+\cdots \in\mathcal{B}_0$ such that
	\begin{equation}\label{fw}
	\frac{zD_q f(z)}{f(z)}=\frac{1+\omega(z)}{1-\omega(z)}.
	\end{equation}
	Comparing the coefficients on both sides of \eqref{fw}, we have
	\begin{align}
	a_2&=\frac{2b_1}{q},\label{a2}\\
	a_3&=\frac{2b_2q+4b_1^2+2b_1^2q}{q^2(1+q)},\label{a3}\\
	a_4&=\frac{2q^2(1+q)b_3+4q(2+2q+q^2)b_1b_2+2(4+4q+3q^2+q^3)b_1^3}{q^3(1+q)(1+q+q^2)}.\label{a4}
	\end{align}
	We may assume \( b_1 \) is non-negative, as \( \omega(z) \) is rotationally invariant. Furthermore, we suppose \( b_1 \in [0, 1] \) since \( |b_1| \leq 1 \). It is easy to observe that
	\begin{equation*}
	|a_2|=\frac{2b_1}{q}\leq\frac{2}{q}.
	\end{equation*}
	The equality holds if and only if $b_1=1$.
	By Remark \ref{wb11}, we have $\omega(z)=z$.
	Applying the parametric presentation of $b_2$ in \eqref{b2b3}, we have
	\begin{align*}
	|a_3|
	=\frac{1}{q^2(1+q)}\left|(2q+4)b_1^2+2q(1-b_1^2)x\right|
	\leq\frac{4+2q}{q^2(1+q)}.
	\end{align*}
	The equality holds if and only if $x=0$ and $b_1=1$.
	In this case $\omega(z)=z$.
	
	For the estimation of $a_4$,  by using the quantity \eqref{a4}, we have
	\begin{align*}
	|a_4|&=\left|\frac{2q^2(1+q)b_3+4q(2+2q+q^2)b_1b_2+2(4+4q+3q^2+q^3)b_1^3}{q^3(1+q)(1+q+q^2)}\right|\\
	&=\frac{2}{q(1+q+q^2)}\cdot\left|b_3+\frac{2(2+2q+q^2)}{q(q+1)}b_1b_2+\frac{4+4q+3q^2+q^3}{q^2(1+q)}b_1^3\right|.
	\end{align*}
	Then, we could apply Lemma \ref{b2} with
	$$\mu=\frac{4+4q+2q^2}{q(1+q)},\qquad \nu=\frac{4+4q+3q^2+q^3}{q^2(1+q)}.$$
It can be observed that \( \mu > 4 \) and \( \nu > \dfrac{2}{3}(\mu - 1) \) given that \( q \in (0, 1) \). By Lemma \ref{b2}, we have
	\begin{align*}
	|a_4|
	&\leq\frac{2(4+4q+3q^2+q^3)}{q^3(1+q+q^2)(1+q)}.
	\end{align*}
	According to the literature \cite{prokhorov1981}, equality holds when $b_1=1$, i.e., $\omega(z)=z$.
	
	When $b_1=1$, we have $\omega(z)=z$, then
	\begin{equation}\label{exfun}
	\frac{zD_{q}f(z)}{f(z)}=\frac{1+z}{1-z}
	\end{equation}
	which is equivalent to
	\begin{align*}
	\frac{f(z)-f(q z)}{(1-q)f(z)}=\frac{1+z}{1-z}.
	\end{align*}
Multiplying both sides by \( (1 - q)f(z) \) and rearranging the terms, we obtain
\begin{equation}
f(q z) = f(z) \cdot \frac{q(1 + z) - 2z}{1 - z}.
\end{equation}
	
	We may assume the solution as an infinite product:
	\begin{equation}
	f(z) = z \prod_{k=0}^\infty \frac{q - q^{k+1} z}{q- q^k (2-q)z}.
	\end{equation}
	For $|q| < 1$, the product converges.
	The proof is completed.
\end{proof}

\begin{remark}
	When $q\to 1^-$, this theorem reduce to the Bieberbach conjecture(de Branges's theorem) for the starlike function, i.e.
$\left|a_2\right|\le2$, $\left|a_3\right|\le3$ and $\left|a_4\right|\le4.$
\end{remark}

\begin{remark}\label{conjecture}
	When $b_1=1$, we have $\omega(z)=z$. In this case, we have equation \eqref{exfun},
	which can be reformulated as:
	\begin{align*}
	f(z)\ast\frac{z}{(1-q z)(1-z)}=f(z)\cdot\frac{1+z}{1-z}.
	\end{align*}
	The above equation could be expressed as
	\begin{align*}
	\sum_{n=1}^{\infty}[n]_{q}a_nz^n=z+(2+a_2)z^2+\cdots+(2(1+a_2+\cdots+a_{n-1})+a_n)z^n+\cdots,
	\end{align*}
	where $[n]_{q}=1+q+q^2+\cdots+q^{n-1}$
	is the $q$-number.
	Comparing the coefficients of both sides, we obtain
	\begin{align*}
	a_2
	=\frac{2}{q},\quad a_3
	=\frac{2(2+q)}{q^2(1+q)},\quad a_4
	=\frac{2(4+4q+3q^2+q^3)}{q^3(1+q)(1+q+q^2)}.
	\end{align*}	
	Furthermore, it is not difficult to obtain
	\begin{equation}
	a_n=\prod_{k=2}^{n}\frac{1+[k-1]_q}{[k]_q-1}.
	\end{equation}
\end{remark}

\section{The Hankel and Toeplitz determinants of $S_q^*$}
The Hankel and Toeplitz determinants are defined by
\begin{equation*}
H_n^{(k)}=\begin{vmatrix}
a_n&a_{n+1}&\cdots&a_{n+k-1}\\
a_{n+1}&a_{n+2}&\cdots&a_{n+k}\\
\vdots&\vdots&&\vdots\\
a_{n+k-1}&a_{n+k}&\cdots&a_{n+2k-2}
\end{vmatrix}\;\text{and}\;
T_n^{(k)}=\begin{vmatrix}
a_n&a_{n+1}&\cdots&a_{n+k-1}\\
a_{n+1}&a_{n}&\cdots&a_{n+k-2}\\
\vdots&\vdots&&\vdots\\
a_{n+k-1}&a_{n+k-2}&\cdots&a_{n}
\end{vmatrix}
\end{equation*}
separately, for $k\geq 1$ and $n\geq 0$.
\begin{theorem}\label{thm3}
	If a function $f(z)=z+\displaystyle\sum_{n=2}^{\infty}a_nz^n\in S_{q}^*$  and $q\in(0,1)$, then
	\begin{equation}\label{a2a3-a4}
	|a_2a_3-a_4|\leq \frac{2(q+2)}{q^2(1+q+q^2)}.
	\end{equation}
	This estimate is sharp.
\end{theorem}
\begin{proof}
By equalities \eqref{a2}-\eqref{a4}, we have
	\begin{align*}
	|a_2a_3-a_4|
	&=\frac{2}{q^2(1+q+q^2)}\left|-q b_3-2b_1b_2+(q+2)b_1^3\right|.
	\end{align*}
Set \( \mu = \dfrac{2}{q} \) and \( \nu = -\dfrac{q + 2}{q} \). It can be easily verified that
\[
\mu \geq 2 > \frac{1}{2}, \quad \nu + \frac{2}{3}(\mu + 1) = -\frac{q + 2}{3q} < 0
\]
for \( q \in (0,1) \). By Lemma \ref{b2}, we thus obtain
\[
\left| -qb_3 - 2b_1b_2 + (q + 2)b_1^3 \right| \leq q \cdot \frac{q + 2}{q} = q + 2.
\]
It follows that
\[
|a_2a_3 - a_4| \leq \frac{2}{q^2(1 + q + q^2)} \cdot (q + 2) = \frac{2(q + 2)}{q^2(1 + q + q^2)}.
\]
According to \cite{prokhorov1981}, equality holds when \( \omega(z) = z \). This completes the proof.
\end{proof}

\begin{remark}
	Obviously,when $q\to 1^-$, the equality \eqref{a2a3-a4} reduce to a result of Cho et al. {\cite{choo2018}} (Theorem 2.1 with $\alpha=0$).
\end{remark}
\begin{theorem}\label{theorem:3}
	If a function $f(z)=z+\displaystyle\sum_{n=2}^{\infty}a_nz^n\in S_{q}^*$ and  $q\in(0,1)$, then
	\begin{align*}
	|H_1^{(2)}|&=|a_3-a_2^2|\leq \frac{2}{q(q+1)},\\
	|H_2^{(2)}|&=|a_2a_4-a_3^2|
	\leq
	\left\{
	\begin{aligned}
	&\frac{4}{q^2(1+q)^2},\qquad  &&(a_2=0), \\
	&\frac{4(2+q)}{q^2(1+q)^2(1+q+q^2)},\quad &&(a_2\neq0).\\
	\end{aligned}
	\right.
	\end{align*}
	These estimates are sharp.
\end{theorem}
\begin{proof}
 For $f(z)=z+\displaystyle\sum_{n=2}^{\infty}a_nz^n\in S_{q}^* $, using equalities \eqref{a2} and \eqref{a3}, we have
	\begin{equation*}
	|H_1^{(2)}|=
	|a_3-a_2^2|=\left|\frac{2b_2+2b_1^2}{q(1+q)}\right|=\left|\frac{2x(1-b_1^2)-2b_1^2}{q(1+q)}\right|\le \frac{2}{q(q+1)},
	\end{equation*}
with the equality holds if and only if $x=-1$. By Remark \ref{wx-1}, in this case we know  $\omega(z)=z(z-b_1)/(b_1z-1)$.
	
Similarly, applying \eqref{a2} and \eqref{a3} to \( H_2^{(2)} \), we obtain
\begin{equation*}\label{a243}
|H_2^{(2)}| = |a_2a_4 - a_3^2| = \frac{4}{q^2(1 + q + q^2)} \left| b_1b_3 - \frac{1 + q + q^2}{(1 + q)^2}b_2^2 + \frac{2}{(1 + q)^2}b_1^2b_2 - \frac{2 + q}{(1 + q)^2}b_1^4 \right|.
\end{equation*}

For \( b_1 = 0 \), we have \( a_2 = 0 \) by \eqref{a2} and \( b_2 = x \) by \eqref{b2b3}. It follows that
\begin{equation*}
|a_2a_4 - a_3^2| = \frac{4}{q^2(1 + q)^2} \left| -x^2 \right| \leq \frac{4}{q^2(1 + q)^2}.
\end{equation*}
Equality holds if and only if \( |x| = 1 \). In this case, by the Carath\'{e}odory-Toeplitz Theorem, the corresponding Carath\'{e}odory function is \( p(z) = \dfrac{1 + x z^2}{1 - x z^2} \), which implies \( \omega(z) = x z^2 \).

For \( b_1 = 1 \), applying \eqref{b2b3} gives \( b_2 = 0 \) and \( b_3 = 0 \). Here,
\begin{equation*}
|a_2a_4 - a_3^2| = \frac{4(2 + q)}{q^2(1 + q)^2(1 + q + q^2)}.
\end{equation*}
	
For \( b_1 \in (0, 1) \), substituting \( b_2 \) and \( b_3 \) via \eqref{b2b3} yields
\begin{align*}
&\left| b_1b_3 - \frac{1 + q + q^2}{(1 + q)^2}b_2^2 + \frac{2}{(1 + q)^2}b_1^2b_2 - \frac{2 + q}{(1 + q)^2}b_1^4 \right| \\
&= \frac{4b_1(1 - b_1^2)}{q^2(1 + q + q^2)} \left| (1 - |x|^2)y + cx^2 + bx + a \right|,
\end{align*}
	where
	\begin{align*}
	a=-\frac{(2+q)b_1^3}{(1-b_1^2)(1+q)^2},\quad b=\frac {2b_1}{(1+q)^2},\quad c=-b_1-\frac{(1+q+q^2)(1-b_1^2)(1+q)^2}{b_1}.
	\end{align*}
	Thus
	$$
	|a_2a_4-a_3^2|\leq
	\frac{4b_1(1-b_1^2)}{q^2(1+q+q^2)}\left((1-|x|^2)+\left|cx^2+bx+a\right|\right)	
	.$$
	By using Lemma \ref{lemmaY2}, we have
	\begin{align*}
	\left|a_2a_4-a_3^2\right|&\leq\frac{4b_1(1-b_1^2)}{q^2(1+q+q^2)}\times \left(\frac{(1-q)b_1}{(1+q)(1-b_1^2)}+\frac{1+q+q^2}{b_1(1-b_1^2)(1+q)^2}\right)\\
	&=\frac{4b_1(1-q)}{q^2(1+q+q^2)(1+q)}+\frac{4}{q^2(q+1)^2}\\
	&\leq \frac{4(1-q)}{q^2(1+q+q^2)(1+q)}+\frac{4}{q^2(q+1)^2}\\
	&=\frac{4(2+q)}{q^2(1+q)^2(1+q+q^2)},
	\end{align*}
	since $q\in (0,1)$ and $b_1\in(0,1)$. The equality holds if and only $b_1=1$, which by Remark \ref{wb11} implies $\omega(z)=z$.
	This completes the proof.
\end{proof}

\begin{remark} We could show that $\dfrac{4}{q^2(1+q)^2}\leq \dfrac{4(2+q)}{q^2(1+q)^2(1+q+q^2)}$ for all $q\in(0,1)$. If we take
	\begin{align}
	h(q)&:=\frac{4(2+q)}{q^2(1+q)^2(1+q+q^2)}-\frac{4}{q^2(1+q)^2},
	\end{align}
	then
	$$h'(q)=\frac{4(-2-5q-4q^2+q^3+4q^4)}{q^3(1+q)^2(1+q+q^2)^2}.$$
Because
\begin{align*}
-2 - 5q - 4q^2 + q^3 + 4q^4 &< -2 - 5q - 4q^2 + q^2 + 4q^2 = q^2 - 5q - 2 < 0,
\end{align*}
for \( q \in (0, 1) \), it follows that \( h'(q) < 0 \); that is, \( h(q) \) is decreasing on the interval \( (0, 1) \). Since \( h(1) = 0 \), we conclude that \( h(q) \) must be positive for all \( q \in (0, 1) \).
	
	Therefore, the bound $|H_2^{(2)}|\leq \dfrac{4}{q^2(1+q)^2}$ established by Agrawal in \cite{Agrawal2021} (Theorem 3.4) is incorrect. 

\end{remark}
\begin{remark}
	When $q\to 1^-$, we obtain
	$$|H_2^{(2)}|=|a_2a_4-a_3^2| \leq 1$$
	for $f(z)=z+\displaystyle\sum_{n=2}^{\infty}a_nz^n$ in the class 
	$S^*$. This result was established by Janteng et al. \cite{janteng2007}.
\end{remark}
\begin{theorem}\label{T123}
	If a function $f(z)=z+\displaystyle\sum_{n=2}^{\infty}a_nz^n\in S_{q}^*$, $q\in(0,1)$, then
	\begin{align*}
	|T_1^{(2)}|&=|a_1^2-a_2^2|\leq 1+\frac{4}{q^{2}},\\
	|T_2^{(2)}|&=|a_2^2-a_3^2|\leq \frac{4}{q^4}+\frac{4(2+q)^2}{q^4(1+q)^2},\\
	|T_3^{(2)}|&=|a_3^2-a_4^2|\leq \frac{4(2+q)^2}{q^4(1+q)^2}+\frac{4(4+4q+3q^2+q^3)^2}{q^6(1+q+q^2)^2(1+q)^2}.
	\end{align*}
	These estimates are sharp.
\end{theorem}
\begin{proof}
By the definition of Toeplitz determinant, we have
	$$|T_1^{(2)}|=|1-a_2^2|\leq 1+|a_2|^2.$$
	Thus, by Theorem \ref{theorem:1}, it follows that
	\begin{equation*}
	|T_1^{(2)}|\leq 1+\frac{4}{q^2}.
	\end{equation*}
	The equality holds if and only if $b_1=1$.
	Similarly, it can be shown that
	\begin{equation}\label{lk}
	|T_2^{(2)}|=|a_2^2-a_3^2|\leq |a_2|^2+|a_3|^2\leq\frac{4}{q^4}+\frac{4(2+q)^2}{q^4(1+q)^2},
	\end{equation}
	according to $|a_2|\leq \dfrac{2}{q}$ and $|a_3|\leq \dfrac{4+2q}{q^2(1+q)}$.  The equalities in \eqref{lk} hold if and only if $b_1=1$ and $\omega(z)=z$.
Applying the same method, we can deduce that
\begin{equation}\label{mn}
|T_3^{(2)}| \leq |a_3|^2 + |a_4|^2 \leq \frac{4(2 + q)^2}{q^4(1 + q)^2} + \frac{4(4 + 4q + 3q^2 + q^3)^2}{q^6(1 + q + q^2)^2(1 + q)^2}
\end{equation}
with equality if and only if \( b_1 = 1 \) and \( \omega(z) = z \).
	The theorem is proved.
\end{proof}

\begin{remark}
	When $q\to 1^-$, we have
	\begin{align*}
	|T_1^{(2)}|=|a_1^2-a_2^2|\leq 5,\quad
	|T_2^{(2)}|=|a_2^2-a_3^2|\leq 13,\quad
	|T_3^{(2)}|=|a_3^2-a_4^2|\leq 25.
	\end{align*}
These represent relevent results concerning starlike functions (see{ \cite{ali2018}}).
\end{remark}
\begin{theorem}\label{T1}
	If a function $f(z)=z+\displaystyle\sum_{n=2}^{\infty}a_nz^n\in S_{q}^*$, $q\in(0,1)$, then
	$$|T_1^{(3)}|
	\leq 1+\frac{8}{q^2}+\frac{4(3q^2+8q+4)}{q^4(1+q)^2}.$$
	This estimate is sharp.
\end{theorem}
\begin{proof}
By the definition of Toeplitz determinant, we have
	\begin{equation*}
	T_1^{(3)}=\begin{vmatrix}
	1&a_2&a_3\\
	a_2&1&a_2\\
	a_3&a_2&1\\
	\end{vmatrix}
	=1-2a_2^2+2a_2^2a_3-a_3^2.
	\end{equation*}
	Thus,
	\begin{align*}
	|T_1^{(3)}|&=|1-2a_2^2+2a_2^2a_3-a_3^2|
	\leq 1+2|a_2|^2+|a_3||a_3-2a_2^2|.
	\end{align*}
	Using equalities \eqref{a2} and \eqref{a3}, we obtain
	\begin{equation}\label{a3a22}
	|a_3-2a_2^2|=\left|\frac{2q b_2+(4+2q)b_1^2}{q^2(1+q)}-\frac{8b_1^2}{q^2}\right|.
	\end{equation}
	Substituting \eqref{b2b3} into the equality \eqref{a3a22} gives	
	\begin{align}
	|a_3-2a_2^2|&=\left|\frac{2q(1-b_1^2)x-(4+6q)b_1^2}{q^2(1+q)}\right|\nonumber\\
	&\leq \frac{2q(1-b_1^2)+(4+6q)b_1^2}{q^2(1+q)}\nonumber\\
	&= \frac{4(1+q)b_1^2+2q}{q^2(1+q)}\nonumber\\
	&\leq \frac{6q+4}{q^2(1+q)}.\label{hc}
	\end{align}
The last inequality becomes an equality if and only if $b_1=1$.
	By Theorem \ref{theorem:1} and inequality \eqref{hc}, we conclude that
	\begin{equation}\label{vb}
	|T_1^{(3)}|\leq 1+\frac{8}{q^2}+\frac{4(3q^2+8q+4)}{q^4(1+q)^2}.
	\end{equation}
Equality in \eqref{vb} holds if and only if \( \omega(z) = z \).
	This completes the proof of the theorem.
\end{proof}

\begin{theorem}\label{T2}
	If a function $f(z)=z+\displaystyle\sum_{n=2}^{\infty}a_nz^n\in S_{q}^*$, $q\in(0,1)$, then
	\begin{equation*}
	|T_2^{(3)}|\leq
	\frac{8 (4 + 4 q + 3 q^2 + 2 q^3 + q^4) (4 + 4 q + 4 q^2 + 3 q^3 + 2 q^4 +
		q^5)}{q^7 (1 + q)^3 (1 + q + q^2)}
	\end{equation*}
	when $a_2=0$, and
	\begin{equation*}
	|T_2^{(3)}|\leq
	\frac{8(4 + 4 q + 4 q^2 + 3 q^3 + 2 q^4 + q^5) (4 + 8 q + 12 q^2 + 9 q^3 +
		5 q^4 + 3 q^5 + q^6)}{q^7 (1 + q)^3 (1 + q + q^2)^2}
	\end{equation*}
	for $a_2\neq 0$.
	These estimates are sharp.
\end{theorem}
\begin{proof}
By the definition of Toeplitz determinant, we obtain
	\begin{equation*}
	T_2^{(3)}=\begin{vmatrix}
	a_2&a_3&a_4\\
	a_3&a_2&a_3\\
	a_4&a_3&a_2\\
	\end{vmatrix}
	=(a_2-a_4)(a_2^2-2a_3^2+a_2a_4),
	\end{equation*}
	and
	\begin{align}\label{T23}
	|T_2^{(3)}|
	\leq \left(|a_2|+|a_4|\right)\left(|a_2|^2+|a_3|^2+|a_2a_4-a_3^2|\right).
	\end{align}
	Applying Theorem \ref{theorem:1} and Theorem \ref{theorem:3} to \eqref{T23}, we readily obtain the conclusion.
	This completes the proof of the theorem.
\end{proof}
\begin{remark}
	When $q\to 1^-$, by Theorem \ref{T1} and \ref{T2}, we have
	\begin{align*}
	|T_1^{(3)}|&=|1-2a_2^2+2a_2^2a_3-a_3^2|\leq 24,\\
	|T_2^{(3)}|&=|(a_2-a_4)(a_2^2-2a_3^2+a_2a_4)|\leq 84.
	\end{align*}
	These results were originally established by Ali et al. in {\cite{ali2018}}.
\end{remark}

\section{The Bieberbach-type problems on $S_{\zeta}^*(\alpha)$}
When the parameter $\zeta$ is a complex number, we have the following results:
\begin{theorem}\label{main}
If we assume \( |\zeta| \leq 1 \), \( \alpha \in [0, 1) \), and \( f(z) = z + \displaystyle\sum_{n=2}^{\infty} a_n z^n \in \mathcal{S}_{\zeta}^{\ast}(\alpha) \), then
\begin{equation*}
|a_n|^2 \leq \frac{\displaystyle\sum_{k=1}^{n-1} \left\{ |(1 - 2\alpha) + [k]_{\zeta}|^2 - |[k]_{\zeta} - 1|^2 \right\} |a_k|^2}{|[n]_{\zeta} - 1|^2}.
\end{equation*}
\end{theorem}
	\begin{proof}
		For $f(z)=z+\displaystyle\sum_{n=2}^{\infty}a_nz^n\in\mathcal{S}_{\zeta}^*(\alpha)$, there exists an analytic function $\omega(z)=b_1z+b_2z^2+\cdots\in B_0$ such that
		\begin{equation}\label{B}
		\frac{\frac{zD_\zeta f(z)}{f(z)}-\alpha}{1-\alpha}=\frac{1+\omega(z)}{1-\omega(z)}.
		\end{equation}
		By comparing the coefficients on both sides, we obtain
		\begin{equation}\label{C}
		\displaystyle\sum_{n=1}^{\infty}[n]_{\zeta}a_nz^n-\displaystyle\sum_{n=1}^{\infty}a_nz^n
		=\omega(z)\left\{(1-2\alpha)\displaystyle\sum_{n=1}^{\infty}a_nz^n+\displaystyle\sum_{n=1}^{\infty}[n]_{\zeta}a_nz^n\right\}
		\end{equation}
		where $a_1=1$. Equation \eqref{C} can be rewritten as
		\begin{equation}
		\displaystyle\sum_{k=1}^{n}([k]_{\zeta}-1)a_kz^k+\displaystyle\sum_{k=n+1}^{\infty}b_kz^k
		=\omega(z)\displaystyle\sum_{k=1}^{n-1}\left\{(1-2\alpha)+[n]_\zeta\right\}a_kz^k.
		\end{equation}
		Hence, we have
		\begin{align}\label{Par}
		\left|\displaystyle\sum_{k=1}^{n}([k]_{\zeta}-1)a_kz^k+\displaystyle\sum_{k=n+1}^{\infty}b_kz^k\right|^2
		\leq \left|\displaystyle\sum_{k=1}^{n-1}\left\{(1-2\alpha)+[k]_\zeta\right\}a_kz^k\right|^2.
		\end{align}
		Applying Parseval's theorem to \eqref{Par}, we obtain
		$$\displaystyle\sum_{k=1}^{n}|[k]_{\zeta}-1|^2|a_k|^2r^{2k}+\displaystyle\sum_{k=n+1}^{\infty}|b_k|^2r^{2k}\leq \displaystyle\sum_{k=1}^{n-1}\left|(1-2\alpha)+[k]_\zeta\right|^2|a_k|^2r^{2k}$$
		Letting $r\rightarrow 1$  and neglecting the terms involving $b_k$, we conclude that
		\begin{equation}\label{D}
		\displaystyle\sum_{k=1}^{n}|[k]_{\zeta}-1|^2|a_k|^2\leq \displaystyle\sum_{k=1}^{n-1}\left|(1-2\alpha)+[k]_\zeta\right|^2|a_k|^2.
		\end{equation}
		A simple transformation of equation \eqref{D} completes the proof of the theorem.
		
	\end{proof}

\begin{theorem}\label{thm:sum}
For positive integers \( k \in \mathbb{N}^+ \), let \( P(k) = \left| (1 - 2\alpha) + [k]_{\zeta} \right|^2 - \left| [k]_{\zeta} - 1 \right|^2 \) and \( Q(k) = \dfrac{(1 - 2\alpha) + [k - 1]_{\zeta}}{[k]_{\zeta} - 1} \). If \( \zeta \) satisfies \( \Re\left( [k]_{\zeta} \right) > \alpha \) for \( \alpha \in (0, 1) \), then
\begin{equation}\label{ab}
\frac{1}{\left| [n]_{\zeta} - 1 \right|^2} \sum_{k=1}^{n-1} P(k) \left| a_k \right|^2 \leq \prod_{k=2}^{n} \left| Q(k) \right|^2.
\end{equation}
	\begin{proof}
For \( n = 2 \), we have \( P(1) = |2 - 2\alpha|^2 \) and \( Q(2) = \dfrac{2 - 2\alpha}{[2]_{\zeta} - 1} \), with
\begin{align*}
\frac{1}{\left|[2]_{\zeta} - 1\right|^2} P(1) = |Q(2)|^2.
\end{align*}
Thus, \eqref{ab} holds for \( n = 2 \) since \( a_1 = 1 \).

Next, assume that inequality \eqref{ab} holds for some integer \( K \geq 2 \), i.e.,
\begin{equation}\label{K}
\frac{1}{\left|[K]_{\zeta} - 1\right|^2} \sum_{k=1}^{K-1} P(k)|a_k|^2 \leq \prod_{k=2}^{K} |Q(k)|^2.
\end{equation}
Given \( \Re[k]_{\zeta} > \alpha \), it is straightforward to show that \( P(k) > 0 \) for all integers \( k \geq 2 \).

Applying Theorem \ref{main}, we obtain
\begin{align*}
\frac{1}{\left|[K+1]_{\zeta} - 1\right|^2} \sum_{k=1}^{K} P(k)|a_k|^2
&= \frac{\sum_{k=1}^{K-1} P(k)|a_k|^2}{\left|[K+1]_{\zeta} - 1\right|^2} + \frac{P(K)|a_K|^2}{\left|[K+1]_{\zeta} - 1\right|^2} \\
&\leq \frac{\sum_{k=1}^{K-1} P(k)|a_k|^2}{\left|[K+1]_{\zeta} - 1\right|^2} \left\{ 1 + \frac{P(K)}{\left|[K]_{\zeta} - 1\right|^2} \right\} \\
&= \frac{\left|(1 - 2\alpha) + [K]_{\zeta}\right|^2}{\left|[K+1]_{\zeta} - 1\right|^2} \cdot \frac{\sum_{k=1}^{K-1} P(k)|a_k|^2}{\left|[K]_{\zeta} - 1\right|^2}.
\end{align*}
By the induction hypothesis, it follows that
\begin{align*}
\frac{1}{\left|[K+1]_{\zeta} - 1\right|^2} \sum_{k=1}^{K} P(k)|a_k|^2 &\leq \prod_{k=2}^{K+1} |Q(k)|^2.
\end{align*}
This completes the proof of the theorem.
	\end{proof}
\end{theorem}

By combining Theorems \ref{main} and \ref{thm:sum}, we obtain the following result.
\begin{theorem}\label{comp}
For \( f \in \mathcal{S}_{\zeta}^{\ast}(\alpha) \), if \( \zeta \) satisfies \( \Re\left( [k]_{\zeta} \right) > \alpha \) for \( \alpha \in [0, 1) \), then we have
\begin{equation}
|a_n| \leq \prod_{k=2}^n \frac{(1 - 2\alpha) + [k - 1]_{\zeta}}{[k]_{\zeta} - 1}.
\end{equation}
\end{theorem}

When $\alpha=0$ and $\zeta$ is real $(\Re[k]_{\zeta}>\alpha$), Theorem \ref{comp} implies the following:

\begin{corollary}\label{theorem:main}
For \( f(z) = z + \displaystyle\sum_{n=2}^{\infty} a_n z^n \in S_{q}^{\ast} \) with \( q \in (0, 1) \), we have
\[
|a_n| \leq \prod_{k=2}^{n} \frac{1 + [k - 1]_{\zeta}}{[k]_{\zeta} - 1}.
\]
This estimate is sharp, and the extremal function is given by
\begin{equation*}
f(z) = z \prod_{k=0}^{\infty} \frac{\zeta - \zeta^{k+1} z}{\zeta - \zeta^k (2 - \zeta) z}.
\end{equation*}
\end{corollary}
\begin{remark}
	It is a result of Agrawal proved in \cite{Agrawal2021} (Theorem 3.3).
\end{remark}
\section{Conclusion}
In this paper, we establish the upper bounds for the Taylor coefficients of $\zeta$-starlike functions of order $\alpha$, where $\zeta$ is a complex parameter and $\alpha\in[0,1)$. Specially when $\alpha=0$ and $\zeta=q\in(0,1)$, we investigate the Hankel and Toeplitz determinants for functions $f\in S^*_{q}(0)$ and  establish their bounds using advanced analytic techniques. Sharp upper bounds for all the coefficients of  $f\in S^*_{q}(0)$ are obtained.

\noindent
{\bf Acknowledgements.}
The authors declare that there is no conflict of interests regarding the publication of this paper.
The authors would like to express their gratitude to Professor Li-Mei Wang for the insightful discussions that have contributed to this work. This work was partly supported by the National Natural Science Foundation of China (Grant numbers 12001063 and 12171055).

\noindent
{\bf Declaration of Generative AI and AI-assisted technologies in the writing process.}
During the preparation of this work the author(s) used [Doubao] in order to achieve a better English presentation. After using this tool/service, the author(s) reviewed and edited the content as needed and take(s) full responsibility for the content of the publication.




\vspace{1cm}\noindent
Ming Li\\
School of Mathematics and Statistics,Changsha University of Science and Technology\\
960, 2nd, Wanjiali RD(S), Hunan, China\\
E-mail:mingli@csust.edu.cn\\
https://orcid.org/0000-0002-2382-2782

\noindent
Ao-Li Zhu,\\
School of Mathematics and Statistics,Changsha University of Science and Technology\\
960, 2nd, Wanjiali RD(S), Hunan, China\\
E-mail:aolizhucl@163.com



\begin{thebibliography}{99}

\bibitem{Ruscheweyh1982}
S. Ruscheweyh, {\it Convolutions in geometric function theory}, Presses de l'Universit\'{e} de Montr\'{e}al,
Qu\'{e}bec, 1982.


\bibitem{piejko2019q}
K. Piejko, J. Sok{\'o}{\l}, and 
K. Tr{\c{a}}bka-Wi{\c{e}}c{\l}aw, {\it On $q$-calculus and starlike functions}, Iran. J. Sci. Technol. Trans. A Sci. {\bf 43} (2019), no. 6, 2879--2883. https://doi.org/10.1007/s40995-019-00758-6 

\bibitem{piejko2020convolution}
K. Piejko and J. Sok{\'o}{\l},
{\it On convolution and $q$-calculus}, Bol. Soc. Mat. Mex. {\bf 26} (2020), no. 2, 349--359.https://doi.org/10.1007/s40590-019-00258-y

\bibitem{piejko2023}
K. Piejko, J. Sok{\'o}{\l}, and
K. Tr{\c{a}}bka-Wi{\c{e}}c{\l}aw,  {\it On $q$-starlike functionsn}, Bull. Sci. Math. {\bf 186} (2023), 103285. https://doi.org/10.1016/j.bulsci.2023.103285

\bibitem{Jackson1909}
F.H. Jackson, {\it {O}n $q$-functions and a certain difference operator}, Trans. R. Soc. of Edinb. {\bf 46}(1909), no. 2, 253--281. https://doi.org/10.1017/S0080456800002751


\bibitem{ismail1990}
M. E. H. Ismail, E. Merkes, and David Styer, {\it A generalization of starlike functions}, Complex Variables. {\bf 14}(1990), no. 1-4, 77--84. https://doi.org/10.1080/17476939008814407


\bibitem{Sarita2017}
S. Agrawal and S.K. Sahoo, {\it A generalization of starlike functions of order alpha}, Hokkaido Math. J. {\bf 46} (2017), no. 1, 15-27. https://doi.org/ 10.14492/hokmj/1498788094



\bibitem{Agrawal2021}
S. Agrawal, {\it On a generalization of starlike functions},
Indag. Math. {\bf 32} (2021), no. 4, 824-832. https://doi.org/10.1016/j.indag.2021.05.002



\bibitem{grenander1958}
U. Grenander and G. Szeg{\"o}, {\it Toeplitz forms and their applications}, Universitatis of California Press, Berkeley and Los Angeles,1958.



\bibitem{pommerenke1966}
C. Pommerenke, {\it On the coefficients and {H}ankel determinants of univalent functions}, J. Lond. Math. Soc. {\bf 1} (1966), no. 1, 111--122. https://doi.org/10.1112/jlms/s1-41.1.111

\bibitem{annals2011}
P. Deift, A. Its, I. Krasovsky, {\em Asymptotics of {T}oeplitz, {H}ankel, and {T}oeplitz+{H}ankel determinants with {F}isher-{H}artwig singularities}, Ann. Math. {\bf 174} (2011), 1243-1299. https://doi.org/http://dx.doi.org/10.4007/annals.2011.174.2.12

\bibitem{janteng2007}
A. Janteng, S. A. Halim, and M. Darus, {\it Hankel determinant for starlike and convex functions}, Int. J. Math. Anal. {\bf 1} (2007), no. 13, 619--625.

\bibitem{choo2018}
N. E. Cho, B. Kowalczyk, and O.S. Kwon, et al., {\it The bounds of some determinants for starlike functions of order
	alpha}, Bull. Malays. Math. Sci. Soc. {\bf 41} (2018), no. 1, 523--535. https://doi.org/10.1007/s40840-017-0476-x

\bibitem{Kwon2019H}
O. S. Kwon, A. Lecko, and Y.J. Sim, {\it The bound of the {H}ankel determinant of the third kind for starlike functions}, Bull. Malays. Math. Sci. Soc. {\bf 42} (2019), no. 2, 767--780. https://doi.org/10.1007/s40840-018-0683-0


\bibitem{zaprawa2021}
P. Zaprawa, M. Obradovi{\'c}, and N. Tuneski, {\it Third {H}ankel determinant for univalent starlike functions}, RACSAM. {\bf 115} (2021), no. 2, 49. https://doi.org/10.1007/s13398-020-00977-2

\bibitem{kowalczyk2022sharp}
B. Kowalczyk, A. Lecko, and D. K. Thomas, {\it The sharp bound of the third {H}ankel determinant for starlike
	functions}, Forum Math. {\bf 34} (2022), no. 5, 1249--1254. https://doi.org/10.1515/forum-2021-0308

\bibitem{rath2022}
B. Rath, K. S. Kumar, and  D. V. Krishna, et al., {\em The sharp bound of the third {H}ankel determinant for starlike functions of order 1/2}, Complex Anal. Oper. Th. {\bf 16} (2022), no. 5, 65. https://doi.org/10.1007/s11785-022-01241-8

\bibitem{sim2022}
Y.J. Sim, D.K. Thomas, and P. Zaprawa, {\it The second {H}ankel determinant for starlike and convex functions of
	order alpha}, Complex Var. Elliptic Equ. {\bf 67} (2022), no. 10, 2423--2443. https://doi.org/10.1080/17476933.2021.1931149

\bibitem{wang2023}
Z. G. Wang; M. Arif,  and Z. H. Liu, et al., {\em Sharp bounds on {H}ankel determinants for certain subclass of starlike functions}, J. Appl. Anal. Comput. {\bf 13} (2023), no.2, 860-873. https://doi.org/10.11948/20220180

\bibitem{Verma2023}
N. Verma, S. S. Kumar, {\em A conjecture on $H_3(1)$ for certain starlike functions}, Math. Slovaca, {\bf 73}, (2023), no. 5, 1197-1206.  https://doi.org/10.1515/ms-2023-0088

\bibitem{Ali2023}
R. M. Ali, S. Kumar, and V. Ravichandran, {\em The third {H}ermitian-{T}oeplitz and {H}ankel determinants for parabolic starlike functions}, Bull. Korean Math. Soc. {\bf 60} (2023), no. 2, 281-291. https://doi.org/10.4134/BKMS.B210368

\bibitem{zaprawa2024}
P. Zaprawa, {\em Hankel determinant $H_{2,3}$ for starlike and convex functions}, Bull. Sci. Math. \textbf{194} (2024), 103459. https://doi.org/10.1016/j.bulsci.2024.103459
 

\bibitem{Kumarverma2025}
S. S. Kumar and N. Verma, {\em On estimation of {H}ankel determinants for certain class of starlike functions}, Filomat, {\bf 39} (2025), no. 12, 3907-3930. https://doi.org/10.2307/27388540


\bibitem{zaprawa2025}
K. Tr\c{a}bka-Wi\c{e}c{\l}aw and P. Zaprawa, {\em Third {H}ankel determinant for the inverse of starlike functions of order $\alpha$}, Bull. Malays. Math. Sci. Soc. {\bf 48} (2025), no. 4, 104. https://doi.org/10.1007/s40840-025-01888-4

\bibitem{srivastava2021}
H. M. Srivastava, B. Khan, and N. Khan, et al., {\it Upper bound of the third {H}ankel determinant for a subclass of
	$q$-starlike functions associated with the $q$-exponential function}, Bull. Sci. Math. {\bf 167} (2021), 102942. https://doi.org/10.1016/j.bulsci.2020.102942


\bibitem{Taj2025}
Y. Taj, S. Zainab, and Q. Xin, Qin et al., {\em Hankel determinants for $q$-starlike functions connected with $q$-sine function}, Demonstr. Math. {\bf 58} (2025), no 1, 20240044. https://doi.org/10.1515/dema-2024-0044
\bibitem{Qiu2025}
J. L. Qiu, Z. G. Wang, and M. Li, {\em Some characterizations for meromorphic 
	$\zeta$-starlike functions}, J. Contemp. Mathemat. Anal. {\bf 60} (2025), 36–47. https://doi.org/10.3103/S1068362324700456
\bibitem{ali2018}
M. F. Ali, D. K. Thomas, and A. Vasudevarao, {\it Toeplitz determinants whose elements are the coefficients of analytic
	and univalent functions}, Bull. Aust. Math. Soc. {\bf 97} (2018), no. 2, 253--264. https://doi.org/10.1017/S0004972717001174




\bibitem{tsuji1959}
M. Tsuji, {\it Potential theory in modern function theory},
Maruzen Co. Ltd., Tokyo, 1959.

\bibitem{libera1982}
{R. J. Libera and E. J. Z{\l}otkiewicz}, {\it Early coefficients of the inverse of a regular convex function}, Proc. Amer. Math. Soc. {\bf 85} (1982), no.2, 225--230. https://doi.org/10.1090/S0002-9939-1982-0652447-5


\bibitem{Simon2005}
B. Simon, {\it Orthogonal polynomials on the unit circle},
American Mathematical Socirty, Providence, 2005.

\bibitem{li2017note}
{M. Li and T. Sugawa}, {\it A note on successive coefficients of convex functions}, Comput. Methods Funct. Theory. {\bf 17} (2017), 179--193. https://doi.org/10.1007/s40315-016-0177-8


\bibitem{prokhorov1981}
{D. V. Prokhorov and J. Szynal}, {\it Inverse coefficients for ($\alpha$, $\beta$)-convex functions}, Ann. Univ. Mariae Curie-Sk{\l}odowska Sect. A. {\bf 35} (1981), no. 15, 125--143. https://doi.org/10.5833/jjgs.26.2160

\bibitem{ohno2018}
{R. Ohno and T. Sugawa}, {\it Coefficient estimates of analytic endomorphisms of the unit disk fixing a point with applications to concave functions}, Kyoto J. Math. {\bf 58} (2018), no. 2, 227--241. https://doi.org/10.1215/21562261-2017-0015




\end{thebibliography}
\end{document}